\documentclass[times,doublespace]{amsart}

\usepackage{graphicx}
\usepackage{amssymb}
\usepackage{epstopdf}

\usepackage[a4paper, total={6in, 9in}]{geometry}

\usepackage{latexsym,amsmath,amsfonts,amsthm,mathrsfs,amssymb,cite}
\usepackage[usenames]{color}
\usepackage{graphicx,epsfig}
\usepackage{amscd}
\usepackage{amsbsy}
\usepackage{bm}
\usepackage{color}
\usepackage{cite}

\newtheorem{thm}{Theorem}[section]

\newtheorem{lem}{Lemma}[section]

\theoremstyle{definition}
\newtheorem{defn}{Definition}[section]

\theoremstyle{remark}

\numberwithin{equation}{section}

\numberwithin{equation}{section}

\newcounter{saveeqn}

\newcommand{\eqnref}[1]{(\ref {#1})}

\newcommand{\Bb}{\mathbf{b}}

\newcommand{\Bx}{\mathbf{x}}
\newcommand{\By}{\mathbf{y}}

\newcommand{\Gl}{\lambda}

\newcommand{\Ocal}{\mathcal{O}}

\newcommand{\RR}{\mathbb{R}}

\newcommand{\beq}{\begin{equation}}
\newcommand{\eeq}{\end{equation}}

\DeclareMathAlphabet{\itbf}{OML}{cmm}{b}{it}

\title[Plasmon modes in layered structures]{On plasmon modes in multi-layer structures}

\author{Xiaoping Fang}
\address{School of Mathematics and Statistics, Hunan University of Technology and Business, Changsha 410205, China; Key Laboratory of Hunan Province for Statistical Learning and Intelligent Computation, Changsha  410205, China; National Key Laboratory of Data Intelligence and Smart Society, Changsha  410205, China}
\email{fxpmath@hutb.edu.cn, fxp1222@163.com}

\author{Youjun Deng$^*$}
\address{School of Mathematics and Statistics, Central South University, Changsha 410083, China}
\email{youjundeng@csu.edu.cn, dengyijun\_001@163.com}
\thanks{$^*$ Corresponding author: youjundeng@csu.edu.cn, dengyijun\_001@163.com}

\date{} 
\begin{document}
\maketitle

\begin{abstract}
In this paper, we consider the plasmon resonance in multi-layer structures. We show that the plasmon mode is equivalent to the eigenvalue problem of a matrix, whose order is the same to the number of layers. For any number of layers, the exact characteristic polynomial is derived by a conjecture and is verified by using induction. It is shown that all the roots to the characteristic polynomial are real and exist in the span $[-1, 2]$, when the background field is uniform in $\RR^3$. Numerical examples are presented for finding all the plasmon modes, and it is surprisingly to find out that such multi-layer structures may induce so called surface-plasmon-resonance-like band.


\end{abstract}

\maketitle

\section{Introduction}
Recently, there have been increasingly great deal of interest in mathematical theory of surface localized resonance (SLR), due to its fundamental basis of many cutting-edge applications such as invisibility cloaking, near-field microscopy, molecular recognition, nano-lithography, and so on, see, e.g., \cite{ACKLM1,ACKLM2,ACKLM3,ADKLMZ,BL2002,BLBCL2006,BS11,CKKL7,HF04,LL15,LLL9,LLL16,Pe99,WN10} and references there in. SLR structure is usually constructed by high contrast material compared with ambient medium, or noble metallic material which may exhibit negative property at some special occasions. Plasmon resonance is the resonant oscillation of conduction electrons at the interface between negative and positive permittivity material caused by the background field.

It is well known that plasmon resonance depends highly on the material structure. In \cite{ADKLMZ}, the authors studied the plasmon resonance for multiple well-separated nanoparticles under the physical model called Drude model, which describes the dependence of material parameter on the frequency of incident wave. Core-shell structures are widely used for analysis of cloaking due to anomalous localized resonance \cite{ACKLM1, ACKLM2, ACKLM3,BS11,CKKL7,WN10}. Plasmon resonance for spherical structure with normal scale has been studied in these years \cite{FDC19,LL15}. It is noted that the above structures are all isotropic material structures. In \cite{DLZ21,DLZ22,RS}, plasmon resonance with anisotropic material structures (nanorod, slender-body) are studied and some sophisticated observations have been investigated. In most of the aforementioned works on plasmon resonance, the spectral of so called Nuemann-Poincar\'e (N-P) type operators are deeply studied, since it may cause the breaking of the invertibility of an integral system derived from related physical partial differential equation. Each eigenvalue is associated with one type of plasmon mode.

We mention that most mathematical formulation of plasmonic structures are structures with homogeneous material parameters. In \cite{FDL15}, the authors studied the plasmon resonance and its heat generation effect of a four-layer structure. As far as we know, this is the first time to mathematically study the plasmon resonance in a piecewise constant material structure. However, the results in \cite{FDL15} can not be generalised to plasmon resonance in any multi-layer structures, which is a quite challenging problem. The focus of this paper is to present a precisely understanding on plasmon resonance in any multi-layer structures. We simply use the conductivity problem for analysis breach, while more sophisticated electro-magnetic system shall be considered in forth coming works. We shall first derive the exact perturbed field in terms of a matrix which contains the material information, together with the structure information. This matrix plays the role of the integral operator in homogeneous material structure case. The explicit formula for the determinant of the matrix, which is equivalent to the characteristic polynomial of another matrix that does not depend on the material parameter, is a stumbling block in the analysis of plasmon modes. After struggling against complicated structure of the matrix, we come up with a conjecture of the exact formulation of characteristic polynomial. This formulation shows that the eigenvalues can be divided by pairs which add up to one. The conjecture is then verified by induction. Besides, the roots of the characteristic polynomial are all real and belong to the span $[-1, 2]$ for uniform background field in $\RR^3$. Such breaking through opens a wide way to analysis of localized resonance in multi-layer structures.

The organization of this paper is as follows. In section 2, we present the mathematical formulation of plasmon resonance in conductivity problem with multi-layer structures.
Some main results on the formulation of characteristic polynomial and its estimation of roots are exposed. We show some similar results for two dimensional case in section 3.
In section 4, numerical examples are presented in finding all the plasmon modes in a fixed multi-layer structure and plasmon resonance is simulated by using physical Drude model in $\RR^3$. Some conclusions are made in section 5. Section 6 is devoted to the proofs of the main results in section 2.

\section{Mathematical formulation and main results}
Consider the conductivity equation
\beq\label{eq:mainmd01}
\left\{
\begin{array}{ll}
\nabla\cdot \varepsilon\nabla u =0, & \mbox{in} \quad \RR^3\\
u-H=\Ocal(|\Bx|^{-1}), & |\Bx|\rightarrow \infty,
\end{array}
\right.
\eeq
where $H$ is the background field which satisfies $\Delta H=0$ in $\RR^3$. The parameter $\varepsilon$ denotes the conductivity which is given by
\beq\label{eq:permidf01}
\varepsilon(\Bx)=\varepsilon_c(\Bx)\chi(D)+\varepsilon_0\chi(\RR^3\setminus\overline{D}),
\eeq
where $D$ is a multi-layer structure and $\chi$ denotes the characteristic function. The values of $\varepsilon_c(\Bx)$ in $D$ are piecewise constants, depending on the number of the layers given. In FIGURE \ref{fig:1}, a six-layer structure is presented.
\begin{figure}[!h]
   \begin{center}
{\includegraphics[width=5in]{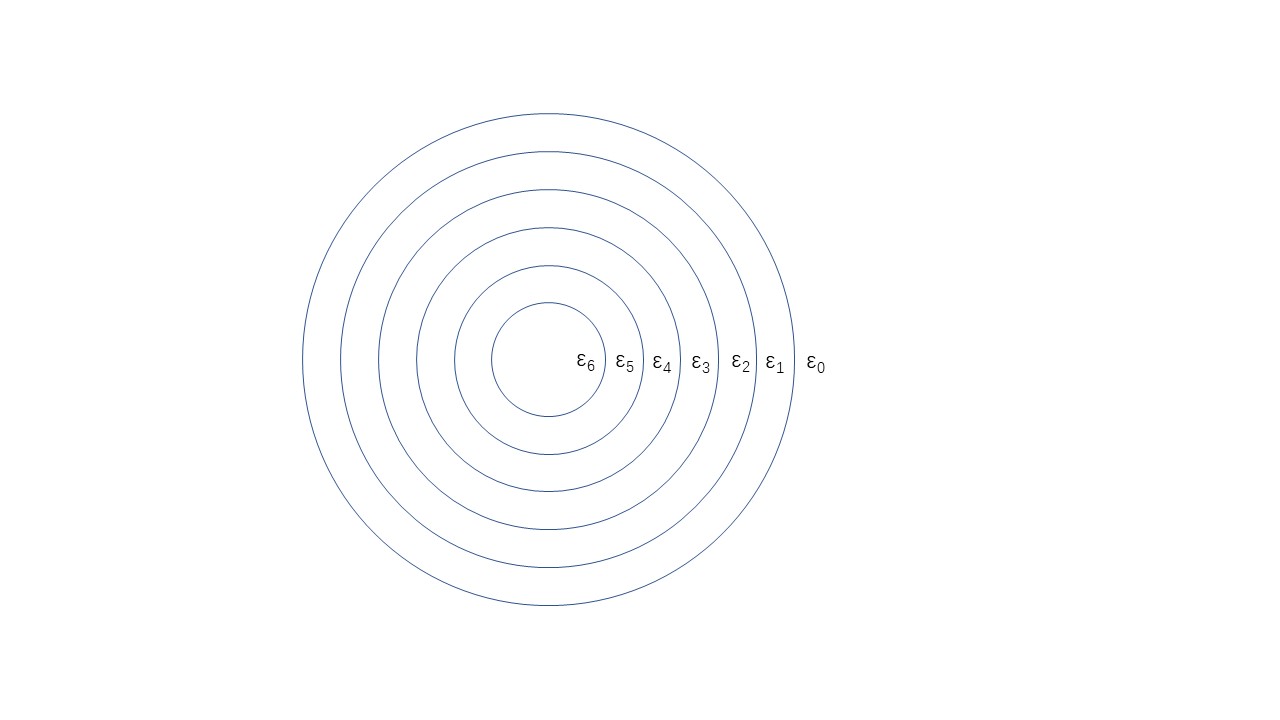}}
   \end{center}
    \caption{Schematic of a six-layer structure.
       }\label{fig:1}
\end{figure}
In general, define the $N$-layer structure by
\beq
\label{eq:aj}
A_{0}:=\{r>r_{1}\}, \quad A_j:=\{r_{j+1}<r\leq r_{j}\}, \quad  j=1,2,\ldots, N-1 \quad A_N:=\{r\leq r_N\},
\eeq
where $N\in \mathbb{N}$.
Assume that
\beq\label{eq:paracho01}
\varepsilon_c(\Bx)=\varepsilon_j, \quad \Bx\in A_j, \quad j=1, 2, \ldots, N.
\eeq
Suppose that the background field is uniformly distributed, that is, the field $H$ can be represented by
\beq\label{eq:defH01}
H=r\sum_{m=-1}^{1}a_{0,m}Y_1^m.
\eeq
By the symmetric properties of the multi-layer structure, we suppose
the total electric potential $u$ has the form
\beq\label{layerstr01}
u=\left\{
\begin{split}
&r\sum_{m=-1}^{1}a_{N,m}Y_1^m, \quad \Bx\in A_N,\\
&r\sum_{m=-1}^{1}a_{j,m}Y_1^m + r^{-2}\sum_{m=-1}^{1}b_{j,m}Y_1^m, \quad \Bx\in A_j,\quad j=N-1, N-2, \ldots , 1\\
&r\sum_{m=-1}^{1}a_{0,m}Y_1^m+ r^{-2}\sum_{m=-1}^{1}b_{0,m}Y_1^m, \quad \Bx\in A_{0}
\end{split}
\right.
\eeq
In what follows, we define
  \beq\label{eq:deflamb01}
 \Gl_{j}=\frac{2\varepsilon_{j-1}+\varepsilon_{j}}{\varepsilon_{j-1}-\varepsilon_{j}}, \quad j=1,2,\ldots,N
 \eeq
The following theorem shows the formula of the perturbed field $u-H$ outside the multi-layer structure.
\begin{thm}\label{th:solmain01}
Suppose $u$ is the solution to \eqnref{eq:mainmd01}, with the parameter $\varepsilon$ given by piecewise constant values in \eqnref{eq:paracho01}, which are positive real numbers. Suppose that $H$ is given by \eqnref{eq:defH01}. Then there holds
\beq\label{eq:purbmn01}
u-H=r^{-3}H\itbf{e}^T \Upsilon_{N} (P_{N}^T)^{-1}\itbf{e},
\eeq
where $\itbf{e}:=(1,1,\ldots,1)^T$, the matrix $P_{N}$ and $\Upsilon_{N}$ are given by
 \beq\label{eq:matP01}
 P_{N}:= \begin{bmatrix}
 \Gl_1 & -1 & -1 & \cdots& -1 \\
2(r_{2}/r_1)^3 & \Gl_{2} & -1 &\cdots & -1 \\
\vdots & \vdots & \vdots & \ddots & \vdots\\
2(r_{N-1}/r_{1})^3 & 2(r_{N-1}/r_{2})^3 & 2(r_{N-1}/r_{3})^3 & \cdots & -1 \\
2(r_{N}/r_1)^3 & 2(r_{N}/r_2)^3 & 2(r_{N}/r_{3})^3 &\cdots & \Gl_{N}
 \end{bmatrix}
 \eeq
and
\beq\label{eq:matU01}
 \Upsilon_{N}:= \begin{bmatrix}
 r_1^3 & 0 & 0 &\cdots& 0 \\
0 & r_{2}^3 & 0 &\cdots& 0 \\
0 & 0 & r_{3}^3 &\cdots& 0 \\
\vdots & \vdots & \vdots & \ddots & \vdots\\
0 & 0 & 0 &\cdots & r_{N}^3
 \end{bmatrix} .
\eeq
\end{thm}

We mention that in general the permittivity $\varepsilon_c$ are positive valued, and thus the matrix $P_{N}$ is invertible. However, in resonance modes, it contains negative values in the multi-layer structure, that is $\varepsilon_j<0$ holds for some $j=1, 2, \ldots, N$. We shall give a precise connection on plasmon resonance and the choice of parameters $\varepsilon_c$ in the multi-layer structure. It can be readily seen that the plasmon modes occur when the matrix $P_N$ satisfy
$|P_N|=0$. One of our aims in this paper is to derive the exact formula for the determinant of $P_N$.
First, we define
 \beq\label{eq:matP02}
 P^i_{M}:= \begin{bmatrix}
\Gl_i & -1 & -1 & \cdots& -1 \\
2(r_{i+1}/r_i)^3 & \Gl_{i+1} & -1 &\cdots & -1 \\
\vdots & \vdots & \vdots & \ddots & \vdots\\
2(r_{M-1}/r_i)^3 & 2(r_{M-1}/r_{i+1})^3 & 2(r_{M-1}/r_{i+2})^3 & \cdots & -1 \\
2(r_{M}/r_i)^3 & 2(r_{M}/r_{i+1})^3 & 2(r_{M}/r_{i+2})^3 &\cdots & \Gl_{M}
 \end{bmatrix}
 .
 \eeq
We also set $P^3_2=1$ in what follows. To simplify the notation we define $t^i_{j}:=(r_{j}/r_{i})^3$, $i, j=1, 2, \ldots N$.  We present an elementary result as follows.
\begin{lem}\label{le:main01}
Suppose $N\geq 4$. The there holds the following recursive relation:
\beq\label{eq:deter0102}
\begin{split}
|P_{N}|=&(\Gl_1+(\Gl_2-1)t^1_{2})(\Gl_{N}+(\Gl_{N-1}-1)t^{N-1}_{N})|P^2_{N-1}| \\
&-(\Gl_1+(\Gl_2-1)t^1_{2})(\Gl_{N-1}+1)(\Gl_{N-1}-2)t^{N-1}_{N}|P^2_{N-2}|\\
&-(\Gl_{2}-2)t^1_{2}(\Gl_{2}+1)(\Gl_{N}+(\Gl_{N-1}-1)t^{N-1}_{N})|P^3_{N-1}|\\
&+(\Gl_{2}-2)t^1_{2}(\Gl_{2}+1)(\Gl_{N-1}+1)(\Gl_{N-1}-2)t^{N-1}_{N}|P^3_{N-2}|.
\end{split}
\eeq
\end{lem}

\subsection{Eigenvalue problem}
It is known that plasmon resonance is usually associated with some eigenvalue problem generated by the PDE system. We shall also explore the related eigenvalue problem for multi-layer structure. To simplify the analysis, we suppose that
\beq\label{eq:vepdef01}
\varepsilon_i
=\left\{
\begin{array}{ll}
-\varepsilon^*+\mathrm{i}\delta, & i \quad \mbox{is odd},\\
\varepsilon_0, & i \quad \mbox{is even},
\end{array}
\right.
\eeq
where $\varepsilon^*$ is a positive number to be chosen and $\delta$ is some small parameter which can be treated as a lossy parameter. $\mathrm{i}=\sqrt{-1}$.
In this setup, one can readily obtain that
$$
\Gl_i=
\left\{
\begin{array}{ll}
\Gl, & i \quad \mbox{is odd},\\
1-\Gl, & i \quad \mbox{is even},
\end{array}
\right.
$$
where
\beq\label{eq:deflamb01}
 \Gl=\frac{2\varepsilon_{0}-\varepsilon^*+\mathrm{i}\delta}{\varepsilon_{0}+\varepsilon^*-\mathrm{i}\delta}, \quad j=1,2,\ldots, N.
\eeq
Then \eqnref{eq:purbmn01} can be rewritten by
\beq\label{eq:eigenspan01}
u-H=r^{-3}H\itbf{e}^T \Upsilon_{N} (\Gl I-K_{N}^T)^{-1} \tilde{\itbf{e}},
\eeq
where $\tilde{\itbf{e}}:=(1,-1, 1, \dots, (-1)^{N-1})^T$ and the matrix $K_N$ is given by
\beq
 K_{N}= \begin{bmatrix}
0 & 1 & -1 & \cdots& (-1)^{N} \\
-2(r_{2}/r_1)^3 & 1 & -1 &\cdots & (-1)^{N} \\
\vdots & \vdots & \vdots & \ddots & \vdots\\
-2(r_{N-1}/r_{1})^3 & 2(r_{N-1}/r_{2})^3 & -2(r_{N-1}/r_{3})^3 & \cdots & (-1)^{N} \\
-2(r_{N}/r_1)^3 & 2(r_{N}/r_2)^3 & -2(r_{N}/r_{3})^3 &\cdots & (1+(-1)^{N})/2
 \end{bmatrix}
.
\eeq
The plasmon resonance mode is parallel to the eigenvalue problem of the matrix $K_N$. In fact, similar to \cite{DLZ21}, we can define the plasmon resonance as follows,
\begin{defn}\label{df:def01}
Consider the system \eqnref{eq:mainmd01} associated with the $N$-layer structure $D$, where the material configuration is described in \eqnref{eq:vepdef01}.
Then plasmon resonance occurs if the following condition is fulfilled:
$$
\lim_{\delta\rightarrow 0} \|\nabla (u-H)\|_{L^2(\RR^3\setminus\overline{D})}=\infty.
$$
\end{defn}
According to Definition \ref{df:def01}, one can readily show that the plasmon modes are the configurations which make the parameter $\Gl$ the eigenvalue of the matrix $K_N$ as $\delta\rightarrow 0$. Thus in what follows, we shall focus on the eigenvalues of $K_N$, or the determinant of $P_N$. To this end, we shall consider the dependence of the determinant $|P^i_m|$, $i\leq m$, on $\Gl$.
First by direct computations one has
\beq\label{eq:rec0101}
\begin{split}
|P^m_{m+1}(\Gl)|=&-(\Gl^2 -\Gl)+2t^{m}_{m+1}, \\
|P^{2m-1}_{2m+1}(\Gl)|=&\left(-(\Gl^2-\Gl)+(2t^{2m-1}_{2m}+2t^{2m}_{2m+1}-2t^{2m-1}_{2m+1})\right)\Gl,\\
|P^{2m}_{2m+2}(\Gl)|=&\Gl^3-2\Gl^2+(2t^{2m}_{2m+2}-2t^{2m}_{2m+1}-2t^{2m+1}_{2m+2}+1)\Gl+2t^{2m}_{2m+1}+2t^{2m+1}_{2m+2}-2t^{2m}_{2m+2}\\
 =& \left(-((1-\Gl)^2-(1-\Gl))+(2t^{2m}_{2m+1}+2t^{2m+1}_{2m+2}-2t^{2m}_{2m+2})\right)(1-\Gl), \quad m\geq 1.
\end{split}
\eeq
With the recursive formula \eqnref{eq:deter0102} and \eqnref{eq:rec0101}, one then has
\beq
\begin{split}
|P_{4}(\Gl)|=&\Gl(1-t^1_{2})(1-\Gl)(1-t^{3}_{4})(\Gl(1-\Gl)+2t^{2}_{3}) -\Gl(1-t^1_{2})(\Gl+1)(\Gl-2)t^{3}_{4}(1-\Gl)\\
&-(\Gl+1)t^1_{2}(\Gl-2)(1-\Gl)(1-t^{3}_{4})\Gl+(\Gl+1)t^1_{2}(\Gl-2)(\Gl+1)(\Gl-2)t^{3}_{4}\\
=&(\Gl^2-\Gl)^2+2(\Gl^2-\Gl)\sum_{(i,j)\in C_4^{0,2}}(-1)^{i+j}t^i_j+4t^1_2t^3_4,
\end{split}
\eeq
where $C^{i,m}_{n}$ denote the set of all combinations of $m$ out $n$, $m\leq n$. For one combination $(i_1, i_2, \ldots, i_{m})\in C^{i,m}_{n}$, $(i+1)\leq i_1,i_2,\ldots i_{m}\leq (n+i)$, we are arranging the order in an ascending way, that is, $i_1<i_2<\cdots< i_{m}$.
Then, in a similar manner, one can show that
\beq
|P_{5}(\Gl)|=\Gl\left((\Gl^2-\Gl)^2+2(\Gl^2-\Gl)\sum_{(i,j)\in C_5^{0,2}}(-1)^{i+j}t^i_j+4\sum_{(i,j,k,l)\in C_5^{0,4}}\tau_{(i,j,k,l)}t^i_jt^k_l\right),
\eeq
where $\tau_{(i,j,k,l)}:=(-1)^{i+j+k+l}$.
Similarly,
\beq
|P_{6}(\Gl)|=-(\Gl^2-\Gl)^3-2(\Gl^2-\Gl)^2\sum_{(i,j)\in C_6^{0,2}}(-1)^{i+j}t^i_j-4(\Gl^2-\Gl)\sum_{(i,j,k,l)\in C_6^{0,4}}\tau_{(i,j,k,l)}t^i_jt^k_l+8t^1_2t^3_4t^5_6.
\eeq
For the sake of simplicity, in what follows, we denote by $\mathbf{i}_{k}$ the multi-index $(i_1,i_2,\ldots,i_{k})$ and $\tau_{\mathbf{i}_{k}}$ its sign given by
\beq
\tau_{\mathbf{i}_{k}}:=(-1)^{\sum_{j=1}^{k} i_j}.
\eeq
By combining the above results, we have following conjecture:
\begin{thm}\label{th:ellresult01}
Define $$q(\lambda):=\Gl^2-\Gl.$$
Then for $N\in \mathbb{N}$, there hold that
\beq\label{eq:maincong01}
|P_N(\Gl)|=
\left\{
\begin{split}
&(-1)^L\left(\sum_{k=0}^L 2^k q(\lambda)^{L-k}\left(\sum_{\mathbf{i}_{2k}\in C_{N}^{0,2k}}\tau_{\mathbf{i}_{2k}}\prod_{l=1}^k t^{i_{2l-1}}_{i_{2l}}\right)\right), &  N=2L,  \\
&\lambda(-1)^L\left(\sum_{k=0}^L 2^k q(\lambda)^{L-k}\left(\sum_{\mathbf{i}_{2k}\in C_{N}^{0,2k}}\tau_{\mathbf{i}_{2k}}\prod_{l=1}^k t^{i_{2l-1}}_{i_{2l}}\right)\right), & N=2L+1.
\end{split}
\right.
\eeq
\end{thm}
We want to mention that if the number of layers, $N$, is odd then there exists one root, i.e., $\lambda=0$ to the determinant $P_N(\Gl)$. This equals to the plasmon mode that $\varepsilon^*=2\varepsilon_0$. The other roots are contained in a quadratic polynomial, whose constant terms are the roots of a polynomial of order $\lfloor N/2\rfloor$, respectively. Here we denote by $\lfloor t \rfloor$ the integer part of $t\in \RR$. In other words, to find the roots of $|P(\Gl)|$, we first find the roots of $|P(q)|$, which is a polynomial of order $\lfloor N/2\rfloor$, with respect to $q$. We then solve the following quadratic equation
$$
\Gl^2-\Gl-q=0
$$
to find the roots of $|P(\Gl)|$.
We shall analyze all the roots to the polynomial $|P_N(\Gl)|$. Note that \eqnref{eq:maincong01} can also be written by the following compact form:
\beq \label{eq:maincong02}
|P_N(\Gl)|=
(-1)^{\lfloor N/2\rfloor}\Gl^{N-2\lfloor N/2\rfloor}\left(\sum_{k=0}^{\lfloor N/2\rfloor} 2^k q(\lambda)^{{\lfloor N/2\rfloor}-k}\left(\sum_{\mathbf{i}_{2k}\in C_{N}^{0,2k}}\tau_{\mathbf{i}_{2k}}\prod_{l=1}^k t^{i_{2l-1}}_{i_{2l}}\right)\right).
\eeq
As mentioned before, to find the plasmon modes, it is essential to find the roots of the polynomial
\beq\label{eq:deffq01}
f_N(q):=\sum_{k=0}^{\lfloor N/2\rfloor} 2^k q^{{\lfloor N/2\rfloor}-k}\left(\sum_{\mathbf{i}_{2k}\in C_{N}^{0,2k}}\tau_{\mathbf{i}_{2k}}\prod_{l=1}^k t^{i_{2l-1}}_{i_{2l}}\right).
\eeq
We have the following elementary result on the roots of $f(q)$:
\begin{thm}\label{th:ellresult02}
Suppose $f_N(q)$ is defined in \eqnref{eq:deffq01}. Then it exists $\lfloor N/2\rfloor$ real values of roots to $f_N(q)$. Besides, let $q^*$ be the root to $f_N(q)$, then
\beq\label{eq:thmainsp01}
q^*\in [-1/4, 2].
\eeq
\end{thm}
With Theorem \ref{th:ellresult02}, one can readily show that the roots of $|P_N(\Gl)|$ are all real values. In fact, for any real solution $q^*$ to \eqnref{eq:deffq01}, one can derive two roots of $|P_N(\Gl)|$, more specifically,
\beq
\Gl=\frac{1\pm\sqrt{1+4q^*}}{2}.
\eeq
By using \eqnref{eq:thmainsp01} one can estimate that
$$
\Gl \in [-1, 2].
$$
\subsection{Extreme case}
In this part, we shall consider that the radius of the layers are extreme large. We mention that the Earth's structure can be treated in this case.
To mathematical describe this case, we set $r_i=R+c_i$, where $R\gg 1$ and $c_i$ are regular constants, $i=1, 2, \ldots, N$. One then has
$$
t^i_j=(r_j/r_i)^3=1+\Ocal(1/R).
$$
Define
\beq\label{eq:coefg01}
g_{N,k}:=\sum_{\mathbf{i}_{k}\in C_{N}^{0,k}}\tau_{\mathbf{i}_{k}},
\eeq
by straight forward computations, one then has
$$
g_{N,1}=\sum_{\mathbf{i}_{1}\in C_{N}^{0,1}}\tau_{\mathbf{i}_{1}}=((-1)^N -1)/2, \quad g_{N,2}=\sum_{\mathbf{i}_{2}\in C_{N}^{0,2}}\tau_{\mathbf{i}_{2}}=- \lfloor N/2\rfloor,
$$
and
$$
g_{2N+1,2N+1}=\sum_{\mathbf{i}_{2N+1}\in C_{2N+1}^{0,2N+1}}\tau_{\mathbf{i}_{2N+1}}=(-1)^{N+1}, \quad g_{N,2\lfloor N/2\rfloor}=\sum_{\mathbf{i}_{2}\in C_{N}^{0,2\lfloor N/2\rfloor}}\tau_{\mathbf{i}_{2\lfloor N/2\rfloor}}=(-1)^{\lfloor N/2\rfloor}.
$$
Besides, one can also observe that
\beq\label{eq:obse01}
g_{2N, 2k-1}=\sum_{\mathbf{i}_{2k-1}\in C_{2N}^{0,2k-1}}\tau_{\mathbf{i}_{2k-1}}=0, \quad k=1, 2, \ldots, N.
\eeq
We have the following recursive formula:
\begin{lem}\label{le:recco01}
Suppose that $g_{N,k}$ is given by \eqnref{eq:coefg01}, then there holds
\beq\label{eq:lerec01}
g_{2N+1,2k}=g_{2N,2k},  \quad k=1, 2, \ldots, N,
\eeq
and
\beq\label{eq:lerec02}
\begin{split}
g_{2N+2,2k}=-g_{2N+1,2k-2}+g_{2N+1,2k}, & \quad k=2, 3, \ldots, N, \\
g_{2N+1,2k-1}=-g_{2N-1,2k-3}+g_{2N-1,2k-1}, & \quad k=2, 3, \ldots, N.
\end{split}
\eeq
\end{lem}
\begin{proof}
First by using \eqnref{eq:obse01} we have
$$
g_{2N+1,2k}=-g_{2N,2k-1}+g_{2N,2k}=g_{2N,2k},
$$
which verifies \eqnref{eq:lerec01}. Next we compute
\[
\begin{split}
g_{2N+1,2k-1}=&-g_{2N, 2k-2}+g_{2N, 2k-1}\\
=&-g_{2N-1,2k-3}+g_{2N-1,2k-1}+g_{2N-1, 2k-1}\\
=&-g_{2N-1, 2k-3}+g_{2N-1, 2k-1},
\end{split}
\]
which verifies the second equation in \eqnref{eq:lerec02}. To prove the first equation in \eqnref{eq:lerec02}, we shall make use of induction. It can be simply verified that the first equation in \eqnref{eq:lerec02} holds for $N\leq 4$. Suppose it holds for $N\leq N_0-1$, $N_0\geq 5$. By combing the above results, one finally has
\[
\begin{split}
g_{2N_0+2,2k}=&g_{2N_0+1, 2k-1}+g_{2N_0+1, 2k}\\
=&-g_{2N_0-1,2k-3}+g_{2N_0-1,2k-1}+g_{2N_0+1, 2k}\\
=&g_{2N_0-2, 2k-4}-g_{2N_0-2,2k-2}+g_{2N_0+1, 2k}\\
=&g_{2N_0-1, 2k-4}-g_{2N_0-1,2k-2}+g_{2N_0+1, 2k}\\
=&-g_{2N_0, 2k-2}+g_{2N_0+1, 2k}=-g_{2N_0+1, 2k-2}+g_{2N_0+1, 2k},
\end{split}
\]
which completes the proof.
\end{proof}
By using Lemma \ref{le:recco01}, the associated polynomial  $f_N(q)$ is then given by
\beq \label{eq:maincongex02}
f_N(q)=
\sum_{k=0}^{\lfloor N/2\rfloor} 2^k q^{{\lfloor N/2\rfloor}-k}g_{N,2k}+\Ocal(1/R),
\eeq
where $g_{N,2k}$ can be calculated recursively by using \eqnref{eq:lerec02}.
By using elementary combination theory, one can show that
\beq
g_{N,2k}=(-1)^kC_{\lfloor N/2 \rfloor}^k,
\eeq
where $C_{\lfloor N/2 \rfloor}^k$ denote the number of combinations for $k$ out of $\lfloor N/2 \rfloor$. \eqnref{eq:maincongex02} then becomes
\beq
f_N(q)=
\sum_{k=0}^{\lfloor N/2\rfloor} (-1)^k2^k q^{{\lfloor N/2\rfloor}-k}C_{\lfloor N/2 \rfloor}^k+\Ocal(1/R)=(q-2)^{\lfloor N/2 \rfloor}+\Ocal(1/R).
\eeq
This indicates that the possible plasmon modes for such set of multi-layer structure are only
$\Gl=-1, 2$ if the number of layers $N$ is even and $\Gl=0, -1, 2$ if $N$ is odd.

\section{Two dimensional case}
For the sake of completeness, we shall present the related results for two dimensional case. Consider again the conductivity problem \eqnref{eq:mainmd01} but in two dimensional space $\RR^2$.
\subsection{Perturbation filed}
Let us keep the notations for the multi-layer structure $D$. Suppose that the background harmonic field $H$ admits the following expansion:
\beq\label{eq:defH0201}
H=\sum_{n=0}^{\infty}\tilde{a}_{0,n}r^n \cos{n\theta},
\eeq
where $r(\cos\theta, \sin\theta)$ is the polar coordinate for $\Bx:=(x, y)$.
Then the solution $u$ to \eqnref{eq:mainmd01} can be represented by
\beq\label{layerstr0201}
u=\left\{
\begin{split}
&\sum_{n=0}^{\infty}\tilde{a}_{N,n}r^n \cos{n\theta}, \quad \Bx\in A_N,\\
&\sum_{n=-\infty}^{\infty}\tilde{a}_{j,n}r^n \cos{n\theta}, \quad \Bx\in A_j,\quad j=N-1, N-2, \ldots , 1\\
&\sum_{n=-\infty}^{\infty}\tilde{a}_{0,n}r^n \cos{n\theta}, \quad \Bx\in A_{0}
\end{split}
\right.
\eeq
By using transmission conditions across each layer $A_j$, $j=1, 2, \ldots, N$, one can then derive that
\beq\label{eq:trans0201}
\left\{
 \begin{split}
 &\tilde{a}_{j,n} r_j^n + \tilde{a}_{j,-n}r_j^{-n} = \tilde{a}_{j-1, n} r_{j}^n + \tilde{a}_{j-1, -n}r_{j}^{-n} , \\
 &\varepsilon_{j} n\left( \tilde{a}_{j, n}r_j^n  -\tilde{a}_{j, -n}r_j^{-n}
\right) = \varepsilon_{j-1} n\left( \tilde{a}_{j-1, n}r_j^n  -\tilde{a}_{j-1, -n} r_{j}^{-n}\right)
 \end{split}
\right.
 \eeq
for $n=0, 1, \ldots, \infty$. By following a similar strategy one thus has the following result:
\begin{thm}\label{th:solmain01}
Suppose $u$ is the solution to
\beq\label{eq:mainmd0201}
\left\{
\begin{array}{ll}
\nabla\cdot \varepsilon\nabla u =0, & \mbox{in} \quad \RR^2\\
u-H=\Ocal(|\Bx|^{-1}), & |\Bx|\rightarrow \infty,
\end{array}
\right.
\eeq
where $H$ is given in \eqnref{eq:defH0201}, with the parameter $\varepsilon$ given by piecewise constant values in \eqnref{eq:paracho01}, which are positive real numbers. Then there holds
\beq\label{eq:purbmn0201}
u-H=\itbf{e}^T\sum_{n=1}^{\infty}r^{-n} \cos{n\theta} \tilde\Upsilon_{N}^{(n)} (\tilde{P}_{N}^{(n)})^{-1}\itbf{e},
\eeq
where $\itbf{e}:=(1,1,\ldots,1)^T$, the matrix $\tilde{P}_{N}^{(n)}$ and $\tilde\Upsilon_{N}^{(n)}$ are given by
 \beq\label{eq:matP0201}
\tilde{P}_{N}^{(n)}:= \begin{bmatrix}
 \tilde\Gl_1 & (r_{2}/r_1)^{2n} & (r_{3}/r_{1})^{2n} & \cdots& (r_{N}/r_1)^{2n} \\
-1& \tilde\Gl_{2} & (r_{3}/r_{2})^{2n} &\cdots &(r_{N}/r_2)^{2n} \\
\vdots & \vdots & \vdots & \ddots & \vdots\\
-1 & -1 & -1  & \cdots & (r_{N}/r_{N-1})^{2n}\\
-1 & -1 & -1 &\cdots & \tilde\Gl_{N}
 \end{bmatrix}
 \eeq
and
\beq\label{eq:matU0201}
\tilde\Upsilon_{N}^{(n)}:= \begin{bmatrix}
 r_1^{2n} & 0 & 0 &\cdots& 0 \\
0 & r_{2}^{2n} & 0 &\cdots& 0 \\
0 & 0 & r_{3}^{2n} &\cdots& 0 \\
\vdots & \vdots & \vdots & \ddots & \vdots\\
0 & 0 & 0 &\cdots & r_{N}^{2n}
 \end{bmatrix} .
\eeq
In \eqnref{eq:matP0201}, the notations $\tilde\Gl_j$, $j=1,2, \ldots, N$ are defined by
  \beq\label{eq:deflamb0201}
 \tilde\Gl_{j}=\frac{\varepsilon_{j-1}+\varepsilon_{j}}{\varepsilon_{j-1}-\varepsilon_{j}}, \quad j=1,2,\ldots,N.
 \eeq
\end{thm}
\subsection{Plasmon modes}
The perturbation formula \eqnref{eq:purbmn0201} exhibits very rich separable resonator information. For example, to analyze the monopole resonator plasmon modes, one only needs to consider the the term for $n=1$ in \eqnref{eq:purbmn0201}, etc. In this part, we shall introduce a similar structure design as for three dimensional case, that is,
we suppose that
\beq\label{eq:vepdef0201}
\varepsilon_i
=\left\{
\begin{array}{ll}
\varepsilon_1, & i \quad \mbox{is odd},\\
\varepsilon_0, & i \quad \mbox{is even},
\end{array}
\right.
\eeq
where $\varepsilon_1$ is some negative material parameter. In such setup, the matrix $\tilde{P}_{N}^{(n)}$ in \eqnref{eq:matP0201} is then reduced to
\beq\label{eq:matP0202}
\tilde{P}_{N}^{(n)}:= \begin{bmatrix}
 \tilde\Gl_1 & (r_{2}/r_1)^{2n} & (r_{3}/r_{1})^{2n} & \cdots& (r_{N}/r_1)^{2n} \\
-1& -\tilde\Gl_{1} & (r_{3}/r_{2})^{2n} &\cdots &(r_{N}/r_2)^{2n} \\
\vdots & \vdots & \vdots & \ddots & \vdots\\
-1 & -1 & -1  & \cdots & (r_{N}/r_{N-1})^{2n}\\
-1 & -1 & -1 &\cdots & (-1)^{N-1}\tilde\Gl_{1}
 \end{bmatrix}.
 \eeq
All the possible plasmon modes for such setup are contained in the solution to
\beq\label{eq:pnn01}
|\tilde{P}_{N}^{(n)}|=0,
\eeq
for any $n\in \mathbb{N}$, which is equivalent to
\beq \label{eq:maincong0201}
\tilde\Gl_{1}^{N-2\lfloor N/2\rfloor}\left(\sum_{k=0}^{\lfloor N/2\rfloor} \tilde\Gl_{1}^{N-2k}\left(\sum_{\mathbf{i}_{2k}\in C_{N}^{0,2k}}\tau_{\mathbf{i}_{2k}}\prod_{l=1}^k \Big(\frac{r_{2l-1}}{r_{2l}}\Big)^2\right)\right)=0.
\eeq
By following a similar arguments as in the proof of Theorem \ref{th:ellresult02}, one can show that
\begin{thm}\label{th:ellresult0201}
There exists $N$ real values of roots to \eqnref{eq:pnn01}. Besides, let $\tilde\Gl^*$ be one root, then
\beq\label{eq:thmainsp01}
\tilde\Gl^*\in [-1, 1].
\eeq
\end{thm}

\section{Numerical examples}
In this section, we shall show some numerical examples to expose our theoretical results. Let us focus on three dimensional case. By the theoretical analysis, to find the plasmon modes for any multi-layer structure, one only needs to find the roots of characteristic polynomial according to Theorem \ref{th:ellresult01}. Moreover, to show the plasmon resonance phenomena, one only needs to observe the energy outside the structure, determined by the term
$$
r^{-3}H\itbf{e}^T \Upsilon_{N} (\Gl I-K_{N}^T)^{-1} \itbf{e}
$$
according \eqnref{eq:eigenspan01} and Definition \ref{df:def01}. We shall investigate both scenarios numerically.
\subsection{Roots to characteristic polynomials}
First, we consider that intervals between each two adjacent layers are equidistance. Let $\varepsilon_0=1$. For $N$-layer structure, set
\beq\label{eq:str01}
r_i=N-i+1, \quad i=1, 2, \ldots N.
\eeq
TABLE \ref{tab:1} show the roots to the characteristic polynomial with $N=19$. In FIGURE \ref{fig:2}, we plot the values of the polynomial $f_N(q)$ in the span $[1.48, 2.0]$ (approximately the span between the third eigenvalue and last eigenvalue), it is interesting that the values of the characteristic polynomial in the span are all very small (of the order $10^{-5}$). This is surprisingly useful in real applications that this span can be used for  surface-plasmon-resonance-like (SPR-like) band.

\begin{table}[h]
\begin{center}
\caption{Roots to the characteristic polynomial with layers which are chosen by \eqnref{eq:str01}. \label{tab:1}}
\begin{tabular}{|c|cccc ccc|}
  \hline
     &  & &  & $N=19$ & &  &  \\
  \hline
  $q$             &   & 1.9794 &  & 1.9664 &  & 1.9457 &  \\
  $\Gl$             & 0 & 1.9931 & -0.9931 & 1.9888 & -0.9888 & 1.9818 & -0.9818  \\
  $\varepsilon_i$   & -2.0000 & -0.0023 & -435.9440 & -0.0038 & -265.9316 & -0.0061 & -163.6370\\
  \hline
$q$&& 1.9112 & & 1.8490 & &1.7294 &  \\
$\Gl$ && 1.9701 & -0.9701 & 1.9488 &-0.9488 & 1.9069 & -0.9069  \\
$\varepsilon_i$ &&  -0.0101 & -99.3122 &  -0.0174 & -57.5786 & -0.0320 & -31.2310 \\
\hline
$q$ && 1.4863 & & 1.0066 &  & 0.3299 &  \\
$\Gl$ && 1.8177 & -0.8177 & 1.6210 & -0.6210 & 1.2615 & -0.2615\\
$\varepsilon_i$&& -0.0647  & -15.4553 & -0.1446 & -6.9153  & -0.3265 & -3.0625\\
\hline
\end{tabular}
\end{center}
\end{table}


\begin{figure}[!h]
   \begin{center}
{\includegraphics[width=2.8in]{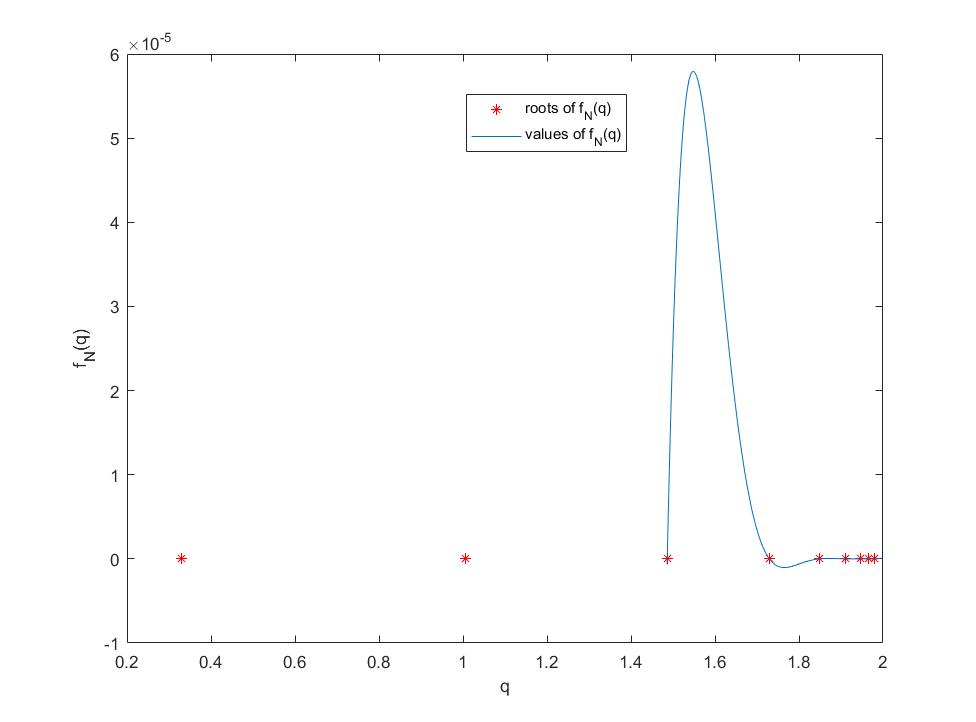}}
   \end{center}
    \caption{Values of $f_N(q)$ in the span [1.48, 2.0], $N=19$.
       }\label{fig:2}
\end{figure}

Next, we consider the radius of layers are decreasing with the same scale $s$, that is
\beq\label{eq:str02}
r_{i+1}=sr_i, \quad i=1, 2, \ldots N-1.
\eeq
Set $r_1=1$ and $s=0.8$. TABLE \ref{tab:2} exhibit all the roots. Similarly, one can also find out that the values in the span $[1.22, 1.79]$ (again approximately the span between the third eigenvalue and last eigenvalue, see FIGURE \ref{fig:3}), are all very small. Besides, it is worth mentioning that in both set up of structures, the roots $q$ are all positive values. Similar results can be found in FIGURE \ref{fig:4} and \ref{fig:5} for $N=11$ and $N=16$, respectively.

\begin{table}[h]
\begin{center}
\caption{Roots to the characteristic polynomial with layers which are chosen by \eqnref{eq:str02}. \label{tab:2}}
\begin{tabular}{|c|cccc ccc|}
  \hline
     &  & &  & $N=19$ & &  &   \\
  \hline
  $q$             &   & 1.7867 &  & 1.7707 &  & 1.7408 & \\
  $\Gl$             & 0 & 1.9271 & -0.9271 & 1.9215 & -0.9215 & 1.9110 & -0.9110 \\
  $\varepsilon_i$   & -2.0000 & -0.0249 & -40.1605 & -0.0269 & -37.2239 & -0.0306 & -32.6949 \\
  \hline
$q$&& 1.6904 & &  1.6073 & &1.4676 &  \\
$\Gl$ && 1.8930 & -0.8930 & 1.8628 & -0.8628 & 1.8106 & -0.8106 \\
$\varepsilon_i$ &&  -0.0370 & -27.0318 &  -0.0479 & -20.8683 & -0.0674 & -14.8369\\
\hline
$q$&& 1.2282 & & 0.8299 &  &0.2982 & \\
$\Gl$ && 1.7158 & -0.7158 & 1.5392 &  -0.5392 & 1.2404 & -0.2404\\
$\varepsilon_i$ & & -0.1046  & -9.5571 & -0.1815 & -5.5104 & -0.3390 &-2.9494\\
\hline
\end{tabular}
\end{center}
\end{table}
\begin{figure}[!h]
   \begin{center}
{\includegraphics[width=2.8in]{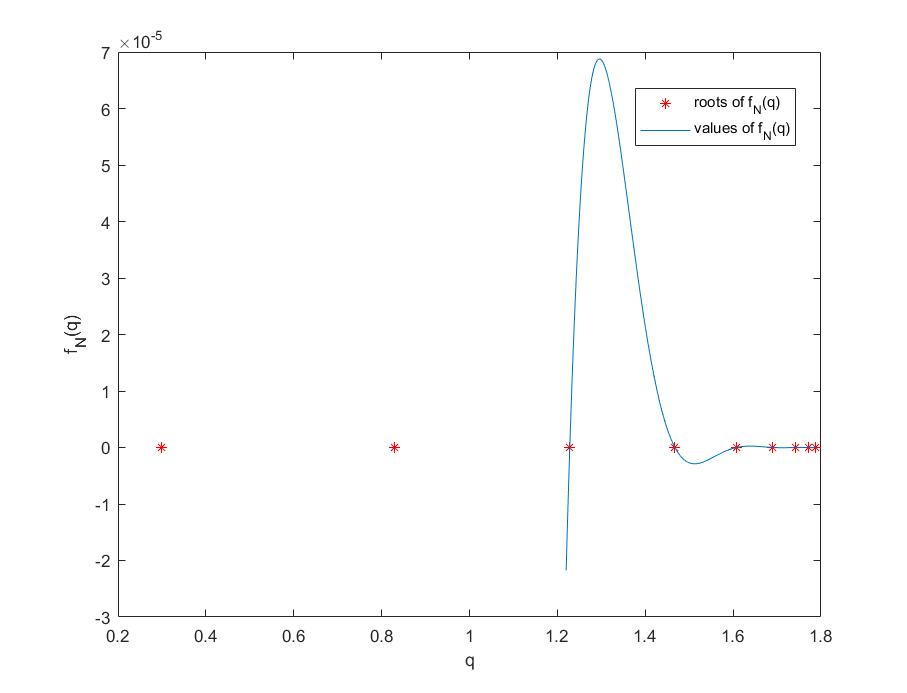}}
   \end{center}
    \caption{Values of $f_N(q)$ in the span [1.22, 1.79], $N=19$.
       }\label{fig:3}
\end{figure}

\begin{figure}[!h]
   \begin{center}
{\includegraphics[width=4.8in]{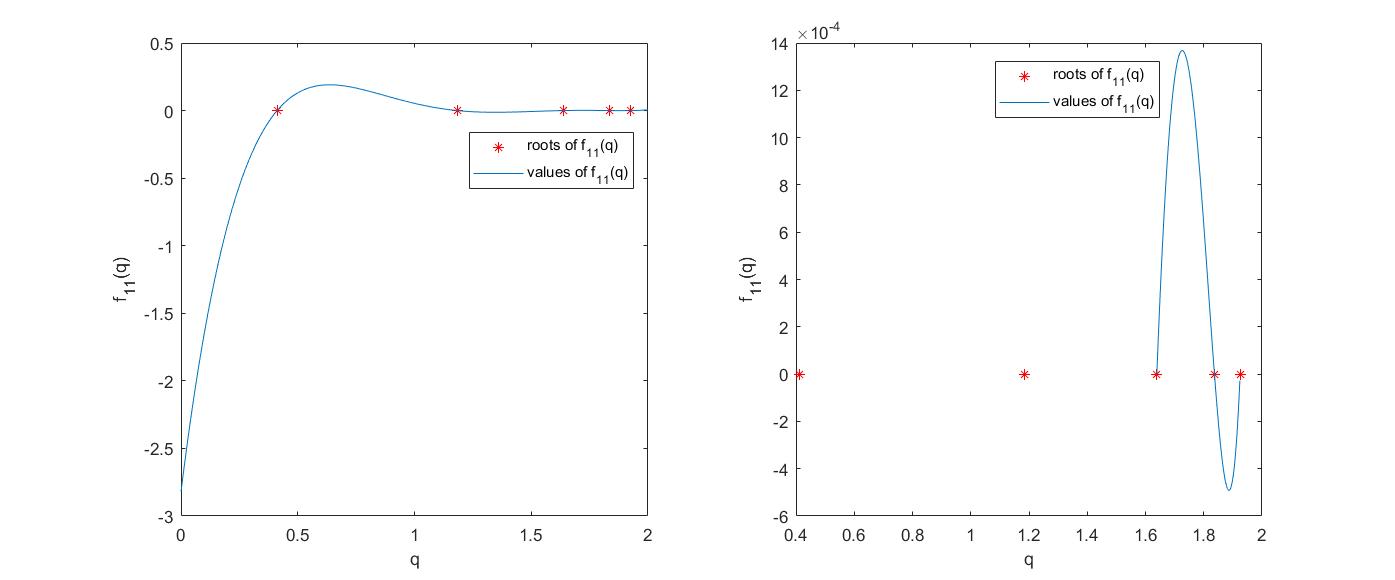}}
   \end{center}
    \caption{Values of $f_N(q)$, $N=11$.
       }\label{fig:4}
\end{figure}

\begin{figure}[!h]
   \begin{center}
{\includegraphics[width=4.8in]{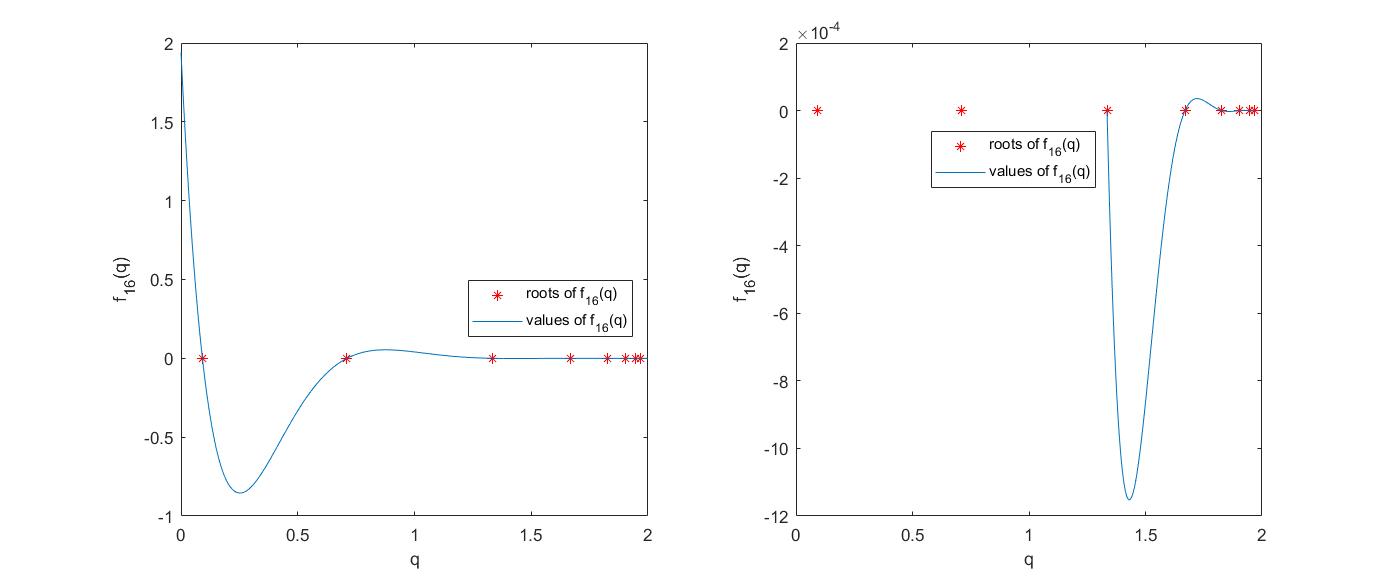}}
   \end{center}
    \caption{Values of $f_N(q)$, $N=16$.
       }\label{fig:5}
\end{figure}
\subsection{Plasmon resonance illustration}
We introduce the Drude model for modeling the parameter $\varepsilon_c$ with the angular frequency $\omega$
(see e.g. \cite{SC10}). In \eqnref{eq:vepdef01}, $\varepsilon_i$ are replaced by
\beq\label{eq:Drude}
\varepsilon_i=\varepsilon'(1-\frac{\omega_p^2}{\omega(\omega+i\tau)}), \quad i \quad \mbox{is odd},
\eeq
where we suppose that (cf. \cite{ADKLMZ})
$$\tau = 10^{14}\, s^{-1}; \varepsilon' = 9 \cdot 10^{-12} F \, m^{-1}; \varepsilon_0 = (1.33)^2 \varepsilon'; \omega_p = 2 \cdot 10^{15} s^{-1}.$$
We define the polarization tensor here by
\beq\label{eq:defpol0101}
\mathbf{M}:=r_1^{-3}\Upsilon_{N} (\Gl I-K_{N}^T)^{-1}.
\eeq
The functionality of $r_1^{-3}$ appearing in \eqnref{eq:defpol0101} is to reduce the scale of the structure.
In FIGURE \ref{fig:6}, we show the norm of polarization tensor defined in \eqnref{eq:defpol0101} with the multi-layer structure designed by \eqnref{eq:str01}, where $N=17$. It can be seen that the peeks of the norm of polarization tensor are in accordance with the plasmon modes in the setup of physical Drude model. The norm of polarization tensor decays as the frequency $\omega$ is getting small. This is due to the existence of lossy parameter $\tau$.
\begin{figure}[!h]
   \begin{center}
{\includegraphics[width=2.8in]{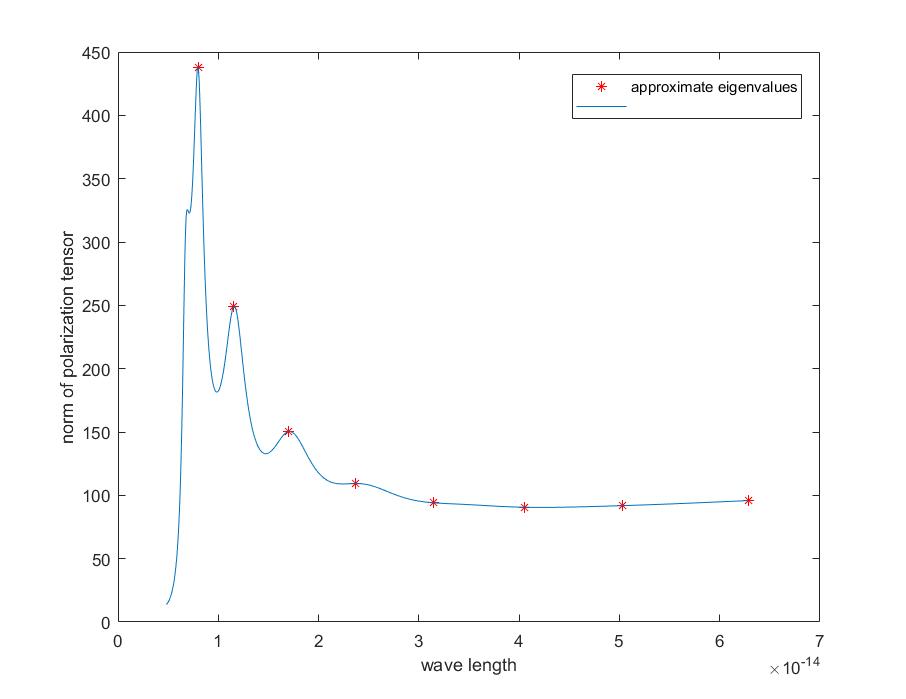}}
   \end{center}
    \caption{Norm of polarization tensor with layers which are chosen by \eqnref{eq:str01}.
       }\label{fig:6}
\end{figure}

\section{Concluding remarks}
We considered plasmon resonance in multi-layer structures. In order to better exhibit the main idea and theoretical results, we
use the simplest conductivity problem associated with uniformly distributed background field for derivation of the matrix $P_N(\Gl)$, together
with its characteristic polynomial. We mention that the idea can be extended to more general electro-magnetic system with some technique adjustments.
We shall derive some similar results regarding the plasmon modes in multi-layer structures for more complicated systems in forthcoming works. Localized resonance of multi-layer structures for elastic problem is another interesting and challenging problem. It is also worthwhile mentioning that the theoretical results can greatly help to find all the plasmon modes for any given multi-layer structures. More numerical observations and justifications will be
implemented, especially the searching for SPR-like band will be verified in more different set up of multi-layer structures. Finally, we also want to mention that multi-layer structures have also been used to generate Generalized Polarization Tensors(GPTs) vanishing structures for enhancement of near cloaking, see e.g., \cite{AGJKLM12,AKLL11,AKLL13}. The authors there designed sophisticated multi-layer structures in order to make the lower orders of GPTs vanishing. However, the existence of the vanishing structures is remain open. The design of multi-layer structures together with the theoretical results in this paper might be a possible way to prove the existence of GPTs vanishing structures.

\section{Proof of main results}
\subsection{Proof of Theorem \ref{th:solmain01}}
The transmission conditions on the interface $\{ r=r_j\}$, $j=1, 2, \ldots N$, imply that
\beq\label{eq:trans01}
\left\{
 \begin{split}
 &a_{j, m} r_j + b_{j, m}r_j^{-2} = a_{j-1, m} r_{j} + b_{j-1, m}r_{j}^{-2} , \\
 &\varepsilon_{j} \left( a_{j, m}  - 2b_{j, m}r_j^{-3}
\right) = \varepsilon_{j-1} \left( a_{j-1, m}  - 2b_{j-1, m} r_{j}^{-3}\right)
 \end{split}
\right.
 \eeq
 where we set $b_{N, m}=0$, $m=-1,0,1$.
By using \eqnref{eq:deflamb01} and some proper arrangements to the equations \eqnref{eq:trans01}, one obtains that
\begin{align*}
&\Gl_1(a_{1,m}-a_{0,m})+2\sum_{j=2}^{N} (a_{j, m}-a_{j-1,m})\Big(\frac{r_j}{r_1}\Big)^3=a_{0, m}, \\
&-\sum_{j=1}^{l-1}(a_{j, m}-a_{j-1,m})+\Gl_l(a_{l,m}-a_{l-1,m})+2\sum_{j=l+1}^{N}(a_{j, m}-a_{j-1,m})\Big(\frac{r_j}{r_l}\Big)^3=a_{0,m}, \quad l=2,3,\ldots,N-1,\\
&-\sum_{j=1}^{N-1}(a_{j, m}-a_{j-1,m})+\Gl_{N}(a_{N,m}-a_{N-1,m})=a_{0,m},
\end{align*}
for $m=-1, 0, 1$.
In the case that the material parameters $\varepsilon_j$, $j=1,2,\ldots N$ are all positive, the matrix $P_{N}$ is invertible and then there holds
\beq\label{eq:a}
\itbf{a}_m = a_{0,m}(\Xi(P_{N}^T)^{-1}\itbf{e} +\itbf{e}),
\eeq
where $\itbf{a}_m:=(a_{1,m},a_{2,m},\ldots,a_{N,m})^T$ and
\beq\label{eq:defxi01}
\Xi=
\begin{bmatrix}
1 & 0 & 0 &\cdots & 0 \\
1 & 1 & 0 &\cdots & 0 \\
\vdots & \vdots &\vdots &\ddots & \vdots\\
1 & 1  & 1 &\cdots & 1
\end{bmatrix}.
\eeq
Furthermore, by the first equation in \eqnref{eq:trans01}, there holds
\beq\label{eq:b}
\Bb_m = a_{0,m}\Xi^T \Upsilon_{N} (P_{N}^T)^{-1} \itbf{e}
\eeq
where $\Bb_m:=(b_{0,m},b_{1,m},  \ldots ,b_{N-1,m})^T$. The proof is complete by extracting the first element in $\Bb_m$.

\subsection{Proof of Lemma \ref{le:main01}}
Denote by $\mathbf{E}_{i,j}$  the elementary matrix which is transform of identity matrix via adding the $j$-th line of identity matrix to the $i$-th line.
Then by using some elementary translation, we compute
\beq
\begin{split}
|P_{N}|= &\left|\mathbf{E}_{N,-(N-1)}\mathbf{E}_{1,-2*t^1_2}P_N\mathbf{E}_{1,-2}\mathbf{E}_{N,(N-1)*t^{N-1}_N}\right|\\
 =&\left|\begin{array}{ccccc}
 \Gl_1+(\Gl_2-1)t^1_{2} & -\Gl_{2}-1 & \cdots& 0& 0 \\
(2-\Gl_{2})t^1_{2} & \Gl_{2} &\cdots& -1 & 0 \\
\vdots & \vdots & \ddots &\vdots & \vdots\\
0 & 2t^2_{2N-1} & \cdots& \Gl_{N-1} & -\Gl_{N-1}-1 \\
0 & 0 &\cdots& (2-\Gl_{N-1})t^{N-1}_{N} & \Gl_{N}+(\Gl_{N-1}-1)t^{N-1}_{N}
 \end{array}
 \right|
\end{split}
\eeq
where we only changed the first line and column, together with the last line and column. By using the notation \eqnref{eq:matP02} and the expansion theorem for determinant one thus has
\beq
\begin{split}
|P_{N}|=&(\Gl_1+(\Gl_2-1)t^1_{2})(\Gl_{N}+(\Gl_{N-1}-1)t^{N-1}_{N})|P^2_{N-1}| \\
&-(\Gl_1+(\Gl_2-1)t^1_{2})(\Gl_{N-1}+1)(\Gl_{N-1}-2)t^{N-1}_{N}|P^2_{N-2}|\\
&-(\Gl_{2}-2)t^1_{2}(\Gl_{2}+1)(\Gl_{N}+(\Gl_{N-1}-1)t^{N-1}_{N})|P^3_{N-1}|\\
&+(\Gl_{2}-2)t^1_{2}(\Gl_{2}+1)(\Gl_{N-1}+1)(\Gl_{N-1}-2)t^{N-1}_{N}|P^3_{N-2}|
\end{split}
\eeq
holds for all $N\geq 4$, $N\in \mathbb{N}$.

\subsection{Proof of Theorem \ref{th:ellresult01}}
We shall make use of induction. It is obvious that $|P_1(\Gl)|=\Gl$ which is contained in \eqnref{eq:maincong01} , by using \eqnref{eq:rec0101}, one has the assertion for $N=2$ and $N=3$. Now we suppose that \eqnref{eq:maincong01} holds for all $N\leq N_0$, $N_0\geq 4$, we show that it also holds for $N=N_0+1$. Note that $q(1-\Gl)=q(\Gl)$, in what follows, we shall write $q$ instead of $q(\Gl)$ for simplicity.
\begin{enumerate}
  \item [Case i)] $N_0$ is even. By making use of \eqnref{eq:deter0102} and \eqnref{eq:maincong01} for all $N\leq N_0$, one has
\beq
\begin{split}
|P_{N}(\Gl)|=&\Gl^2(1-t^1_{2})(1-t^{N-1}_{N})|P^2_{N-1}(1-\Gl)| +\Gl(1-t^1_{2})(2-\Gl)(\Gl+1)t^{N-1}_{N}|P^2_{N-2}(1-\Gl)|\\
&+t^1_{2}(2-\Gl)(\Gl+1)\Gl(1-t^{N-1}_{N})|P^3_{N-1}(\Gl)|+(2-\Gl)^2(\Gl+1)^2t^1_{2}t^{N-1}_{N}|P^3_{N-2}(\Gl)|\\
=&\sum_{j=1}^4 d_j \mathcal{D}_j,
\end{split}
\eeq
where we use the notations
\[
\begin{split}
&d_1=\Gl^2(1-t^1_{2})(1-t^{N-1}_{N})(-1)^{N_0/2-1}(1-\Gl), \quad d_2=\Gl(1-t^1_{2})(2-\Gl)(\Gl+1)t^{N-1}_{N}(-1)^{N_0/2-1}, \\
&d_3=t^1_{2}(2-\Gl)(\Gl+1)\Gl(1-t^{N-1}_{N})(-1)^{N_0/2-1}, \quad d_4=(2-\Gl)^2(\Gl+1)^2t^1_{2}t^{N-1}_{N}(-1)^{N_0/2}\Gl,
\end{split}
\]
and
\[
\mathcal{D}_j=\left(\sum_{k=0}^{N_0/2-1-\lfloor j/4\rfloor} 2^k q^{N_0/2-k-1-\lfloor j/4\rfloor}\left(\sum_{\mathbf{i}_{2k}\in C_{N-2-\lfloor j/2\rfloor}^{\lfloor (j+1)/2\rfloor,2k}}\tau_{\mathbf{i}_{2k}}\prod_{l=1}^k t^{i_{2l-1}}_{i_{2l}}\right)\right).
\]
By setting
\[
\begin{split}
&\mathfrak{a}_k:=\sum_{\mathbf{i}_{2k}\in C_{N-2}^{1,2k}}\tau_{\mathbf{i}_{2k}}\prod_{l=1}^k t^{i_{2l-1}}_{i_{2l}},
\quad \mathfrak{b}_k:=\sum_{\mathbf{i}_{2k}\in C_{N-3}^{1,2k}}\tau_{\mathbf{i}_{2k}}\prod_{l=1}^k t^{i_{2l-1}}_{i_{2l}},\\
&\mathfrak{c}_k:=\sum_{\mathbf{i}_{2k}\in C_{N-3}^{2,2k}}\tau_{\mathbf{i}_{2k}}\prod_{l=1}^k t^{i_{2l-1}}_{i_{2l}},
\quad \mathfrak{d}_k:=\sum_{\mathbf{i}_{2k}\in C_{N-4}^{2,2k}}\tau_{\mathbf{i}_{2k}}\prod_{l=1}^k t^{i_{2l-1}}_{i_{2l}}.
\end{split}
\]
and some straight forward computations one has
\[
\begin{split}
|P_{N}(\Gl)|=&(-1)^{N_0/2}\Gl\Big(q(1-t^1_{2})(1-t^{N-1}_{N})\left(\sum_{k=0}^{N_0/2-1} 2^k q^{N_0/2-k-1}\mathfrak{a}_k\right)\\
&\quad\quad\quad\quad\quad+(q-2)(1-t^1_{2})t^{N-1}_{N}\left(\sum_{k=0}^{N_0/2-1} 2^k q^{N_0/2-k-1}\mathfrak{b}_k\right)\\
&\quad\quad\quad\quad\quad+(q-2)t^1_{2}(1-t^{N-1}_{N})\left(\sum_{k=0}^{N_0/2-1} 2^k q^{N_0/2-k-1}\mathfrak{c}_k\right)\\
&\quad\quad\quad\quad\quad+(q-2)^2t^1_{2}t^{N-1}_{N}\left(\sum_{k=0}^{N_0/2-2} 2^k q^{N_0/2-k-2}\mathfrak{d}_k\right)\Big)\\
=:&(-1)^{N_0/2}\Gl\left(\sum_{k=0}^{N_0/2}2^kq^{N_0/2-k}g_k\right),
\end{split}
\]
where it can be seen that
\beq
g_0=(1+t^1_{2}t^{N-1}_{N}-t^1_{2}-t^{N-1}_{N})+(1-t^1_{2})t^{N-1}_{N}+t^1_{2}(1-t^{N-1}_{N})+t^1_{2}t^{N-1}_{N}=1,
\eeq
and
\beq
\begin{split}
g_1=&(1+t^1_{2}t^{N-1}_{N}-t^1_{2}-t^{N-1}_{N})\left(\sum_{\mathbf{i}_1\in C_{N-2}^{1,2}}\tau_{\mathbf{i}_1}t^{i_{1}}_{i_{2}}\right)+(1-t^1_{2})t^{N-1}_{N}\left(\sum_{\mathbf{i}_1\in C_{N-3}^{1,2}}\tau_{\mathbf{i}_1}t^{i_{1}}_{i_{2}}\right)\\
&+t^1_{2}(1-t^{N-1}_{N})\left(\sum_{\mathbf{i}_1\in C_{N-3}^{2,2}}\tau_{\mathbf{i}_1}t^{i_{1}}_{i_{2}}\right)+t^1_{2}t^{N-1}_{N}\left(\sum_{\mathbf{i}_1\in C_{N-4}^{2,2}}\tau_{\mathbf{i}_1}t^{i_{1}}_{i_{2}}\right)\\
&-(1-t^1_{2})t^{N-1}_{N}-t^1_{2}(1-t^{N-1}_{N})-2t^1_{2}t^{N-1}_{N}\\
=&\sum_{\mathbf{i}_1\in C_{N}^{0,2}}\tau_{\mathbf{i}_1}t^{i_{1}}_{i_{2}},
\end{split}
\eeq
by using the relations
\[
\begin{split}
\mathfrak{a}_1=\mathfrak{b}_1+\sum_{j=2}^{N-2}(-1)^{j+N-1}t^j_{N-1}, \quad \mathfrak{a}_1=\mathfrak{c}_1+\sum_{j=3}^{N-1}(-1)^{2+j}t^2_j,\\
\mathfrak{b}_1=\mathfrak{d}_1+\sum_{j=3}^{N-2}(-1)^{2+j}t^2_j, \quad \mathfrak{c}_1=\mathfrak{d}_1+\sum_{j=3}^{N-2}(-1)^{j+N-1}t^j_{N-1}.
\end{split}
\]
One can also find out that
\[
\begin{split}
g_{N_0/2}=&-(1-t^1_{2})t^{N-1}_{N}b_{N_0/2-1}-t^1_{2}(1-t^{N-1}_{N})c_{N_0/2-1}+t^1_2t^{N-1}_Nd_{N_0/2-2}\\
=&-(1-t^1_{2})t^{N-1}_{N}\left(\sum_{\mathbf{i}_{N_0/2-1}\in C_{N-3}^{1,N_0-2}}\tau_{\mathbf{i}_{N_0/2-1}}\prod_{l=1}^{N_0/2-1} t^{i_{2l-1}}_{i_{2l}}\right)\\
&-t^1_{2}(1-t^{N-1}_{N})\left(\sum_{\mathbf{i}_{N_0/2-1}\in C_{N-3}^{2,N_0-2}}\tau_{\mathbf{i}_{N_0/2-1}}\prod_{l=1}^{N_0/2-1} t^{i_{2l-1}}_{i_{2l}}\right)\\
&+t^1_2t^{N-1}_N\left(\sum_{\mathbf{i}_{N_0/2-2}\in C_{N-4}^{2,N_0-4}}\tau_{\mathbf{i}_{N_0/2-2}}\prod_{l=1}^{N_0/2-2} t^{i_{2l-1}}_{i_{2l}}\right)\\
=&(t^1_{2}t^{N-1}_{N}-t^{N-1}_{N})(-1)^{N_0/2-1}t^2_3t^4_5\cdots t^{N-3}_{N-2}+(t^1_{2}t^{N-1}_{N}-t^{1}_{2})(-1)^{N_0/2-1}t^3_4t^5_6\cdots t^{N-2}_{N-1}\\
&+t^1_2t^{N-1}_N\left(\sum_{\mathbf{i}_{N_0/2-2}\in C_{N-4}^{2,N_0-4}}\tau_{\mathbf{i}_{N_0/2-2}}\prod_{l=1}^{N_0/2-2} t^{i_{2l-1}}_{i_{2l}}\right)\\
=&\sum_{\mathbf{i}_{N_0/2}\in C_{N}^{0,N_0}}\tau_{\mathbf{i}_{N_0/2}}\prod_{l=1}^{N_0/2} t^{i_{2l-1}}_{i_{2l}}.
\end{split}
\]
Furthermore, one can readily derive that
\[
\begin{split}
g_k=&(1+t^1_{2}t^{N-1}_{N}-t^1_{2}-t^{N-1}_{N})\mathfrak{a}_k+(1-t^1_{2})t^{N-1}_{N}(\mathfrak{b}_k-\mathfrak{b}_{k-1})\\
&+t^1_{2}(1-t^{N-1}_{N})(\mathfrak{c}_k-\mathfrak{c}_{k-1})+t^1_{2}t^{N-1}_{N}(\mathfrak{d}_k-2\mathfrak{d}_{k-1}+\mathfrak{d}_{k-2}),
\end{split}
\]
for all $k=2, 3, \ldots, N_0/2-1$.
Thus there holds
\beq\label{eq:solpfmn01}
\begin{split}
g_k=&(1+t^1_{2}t^{N-1}_{N}-t^1_{2}-t^{N-1}_{N})\left(\mathfrak{a}_k-\mathfrak{c}_k-\mathfrak{b}_k+\mathfrak{d}_k\right)\\
&+(1-t^1_{2})\left(\mathfrak{b}_k-\mathfrak{d}_k\right)+(1-t^{N-1}_{N})\left(\mathfrak{c}_k-\mathfrak{d}_k\right)+\mathfrak{d}_k\\
&+(t^1_2t^{N-1}_N-t^{N-1}_N)\left(\mathfrak{b}_{k-1}-\mathfrak{d}_{k-1}\right)+(t^1_2t^{N-1}_N-t^{1}_2)\left(\mathfrak{c}_{k-1}-\mathfrak{d}_{k-1}\right)\\
&-(t^1_2+t^{N-1}_N)\mathfrak{d}_{k-1}+t^1_{2}t^{N-1}_{N}\mathfrak{d}_{k-2},
\end{split}
\eeq
for $k=2, 3, \ldots, N_0/2-1$.
By straight forward computations one has
\beq\label{eq:solpfmn02}
\mathfrak{b}_k-\mathfrak{d}_k=\sum_{\mathbf{i}_{2k-1}\in C_{N-4}^{2,2k-1}}\tau_{\mathbf{i}_{2k-1}}t^2_{i_{1}}\prod_{l=1}^{k-1} t^{i_{2l}}_{i_{2l+1}},
\quad \mathfrak{c}_k-\mathfrak{d}_k=\sum_{\mathbf{i}_{2k-1}\in C_{N-4}^{2,2k-1}}\tau_{\mathbf{i}_{2k-1}}\prod_{l=1}^{k-1} t^{i_{2l-1}}_{i_{2l}}t^{i_{2k-1}}_{N-1},
\eeq
and
\beq\label{eq:solpfmn03}
\mathfrak{a}_k-\mathfrak{c}_k-\mathfrak{b}_k+\mathfrak{d}_k=\sum_{\mathbf{i}_{2k-2}\in C_{N-4}^{2,2k-2}}\tau_{\mathbf{i}_{2k-2}}t^2_{i_{1}}\prod_{l=1}^{k-2} t^{i_{2l}}_{i_{2l+1}}t^{i_{2k-2}}_{N-1}.
\eeq
By substituting \eqnref{eq:solpfmn02} and \eqnref{eq:solpfmn03} into \eqnref{eq:solpfmn01} and simply computations,  one thus obtains
\beq
\begin{split}
g_k=&\mathfrak{d}_k+\sum_{\mathbf{i}_{2k-2}\in C_{N-4}^{2,2k-2}}(-1)^{N+1}\tau_{\mathbf{i}_{2k-2}}t^2_{i_{1}}\prod_{l=1}^{k-2} t^{i_{2l}}_{i_{2l+1}}t^{i_{2k-2}}_{N-1}+\sum_{\mathbf{i}_{2k-2}\in C_{N-4}^{2,2k-2}}(-1)^{N+1}\tau_{\mathbf{i}_{2k-2}}t^1_{i_{1}}\prod_{l=1}^{k-2} t^{i_{2l}}_{i_{2l+1}}t^{i_{2k-2}}_{N}\\
&\quad+\sum_{\mathbf{i}_{2k-2}\in C_{N-4}^{2,2k-2}}(-1)^N\tau_{\mathbf{i}_{2k-2}}t^1_{i_{1}}\prod_{l=1}^{k-2} t^{i_{2l}}_{i_{2l+1}}t^{i_{2k-2}}_{N-1}+\sum_{\mathbf{i}_{2k-2}\in C_{N-4}^{2,2k-2}}\tau_{\mathbf{i}_{2k-2}}(-1)^{N+2}t^2_{i_{1}}\prod_{l=1}^{k-2} t^{i_{2l}}_{i_{2l+1}}t^{i_{2k-2}}_{N}
\end{split}
\eeq
\[
\begin{split}
&+\sum_{\mathbf{i}_{2k-1}\in C_{N-4}^{2,2k-1}}(-1)^2\tau_{\mathbf{i}_{2k-1}}t^2_{i_{1}}\prod_{l=1}^{k-1} t^{i_{2l}}_{i_{2l+1}}+\sum_{\mathbf{i}_{2k-1}\in C_{N-4}^{2,2k-1}}(-1)\tau_{\mathbf{i}_{2k-1}}t^1_{i_{1}}\prod_{l=1}^{k-1} t^{i_{2l}}_{i_{2l+1}}\\
&+\sum_{\mathbf{i}_{2k-1}\in C_{N-4}^{2,2k-1}}(-1)^{N-1}\tau_{\mathbf{i}_{2k-1}}\prod_{l=1}^{k-1} t^{i_{2l-1}}_{i_{2l}}t^{i_{2k-1}}_{N-1}+\sum_{\mathbf{i}_{2k-1}\in C_{N-4}^{2,2k-1}}(-1)^N\tau_{\mathbf{i}_{2k-1}}\prod_{l=1}^{k-1} t^{i_{2l-1}}_{i_{2l}}t^{i_{2k-1}}_{N}
\end{split}
\]
\[
\begin{split}
\quad\quad\quad&+\sum_{\mathbf{i}_{2k-3}\in C_{N-4}^{2,2k-3}}(-1)^{2N}\tau_{\mathbf{i}_{2k-3}}t^1_{i_{1}}\prod_{l=1}^{k-2} t^{i_{2l}}_{i_{2l+1}}t^{N-1}_{N}+\sum_{\mathbf{i}_{2k-3}\in C_{N-4}^{2,2k-3}}(-1)^{2N+1}\tau_{\mathbf{i}_{2k-3}}t^2_{i_{1}}\prod_{l=1}^{k-2} t^{i_{2l}}_{i_{2l+1}}t^{N-1}_{N}\\
\quad\quad\quad&+\sum_{\mathbf{i}_{2k-3}\in C_{N-4}^{2,2k-3}}(-1)^{N+3}\tau_{\mathbf{i}_{2k-3}}t^1_{2}\prod_{l=1}^{k-2} t^{i_{2l-1}}_{i_{2l}}t^{i_{2k-1}}_{N}
+\sum_{\mathbf{i}_{2k-3}\in C_{N-4}^{2,2k-3}}(-1)^{N+2}\tau_{\mathbf{i}_{2k-3}}t^1_{2}\prod_{l=1}^{k-2} t^{i_{2l-1}}_{i_{2l}}t^{i_{2k-1}}_{N-1}
\end{split}
\]
\[
\begin{split}
&+\sum_{\mathbf{i}_{2k-2}\in C_{N-4}^{2,2k-2}}(-1)^3\tau_{\mathbf{i}_{2k-2}}t^1_2\prod_{l=1}^{k-1} t^{i_{2l-1}}_{i_{2l}}+\sum_{\mathbf{i}_{2k-2}\in C_{N-4}^{2,2k-2}}(-1)^{2N-1}\tau_{\mathbf{i}_{2k-2}}\prod_{l=1}^{k-1} t^{i_{2l-1}}_{i_{2l}}t^{N-1}_N\\
&+\sum_{\mathbf{i}_{2k-4}\in C_{N-4}^{2,2k-4}}(-1)^{2N+2}\tau_{\mathbf{i}_{2k-4}}t^1_2\prod_{l=1}^{k-2} t^{i_{2l-1}}_{i_{2l}}t^{N-1}_N\\
=&\sum_{\mathbf{i}_{2k}\in C_{N}^{0,2k}}\tau_{\mathbf{i}_{2k}}\prod_{l=1}^k t^{i_{2l-1}}_{i_{2l}},
\end{split}
\]
which competes the proof for the case that $N_0$ is even.
  \item [Case ii)] $N_0$ is odd. The proof is similar to the case that $N_0$ is even. We shall only sketch the main ingredients. First, it can be derived that
\[
\begin{split}
|P_{N}(\Gl)|=&(-1)^{N/2}\Big(q(1-t^1_{2})(1-t^{N-1}_{N})\left(\sum_{k=0}^{N/2-1} 2^k q^{N/2-k-1}\mathfrak{a}_k\right)\\
&\quad\quad\quad\quad\quad+q(q-2)(1-t^1_{2})t^{N-1}_{N}\left(\sum_{k=0}^{N/2-2} 2^k q^{N/2-k-2}\mathfrak{b}_k\right)\\
&\quad\quad\quad\quad\quad+q(q-2)t^1_{2}(1-t^{N-1}_{N})\left(\sum_{k=0}^{N/2-2} 2^k q^{N/2-k-2}\mathfrak{c}_k\right)\\
&\quad\quad\quad\quad\quad+(q-2)^2t^1_{2}t^{N-1}_{N}\left(\sum_{k=0}^{N/2-2} 2^k q^{N/2-k-2}\mathfrak{d}_k\right)\Big)\\
=&(-1)^{N/2}\left(\sum_{k=0}^{N/2}2^kq^{N/2-k}h_k\right).
\end{split}
\]
Similarly, one can verify that $h_0=1$ and
$$h_1=\sum_{\mathbf{i}_1\in C_{N}^{0,2}}\tau_{\mathbf{i}_1}t^{i_{1}}_{i_{2}}, \quad h_{N/2}=t^1_2t^3_4\ldots t^{N-1}_N.$$
One can readily find out that $h_k$, $k=2, 3, \ldots, N/2-1$ have the same form with \eqnref{eq:solpfmn01}, thus there holds
$$
h_k=\sum_{\mathbf{i}_{2k}\in C_{N}^{0,2k}}\tau_{\mathbf{i}_{2k}}\prod_{l=1}^k t^{i_{2l-1}}_{i_{2l}}, \quad k=2, 3, \ldots, N/2-1,
$$
which completes the proof.
\end{enumerate}
\subsection{Proof of Theorem \ref{th:ellresult02}}
We assume on contrary that
$$q^*\not\in [-1/4,2], \quad \mbox{or} \quad \Im {q^*}\neq 0.$$
Note that $\Gl^2-\Gl=q^*$, it is equivalent that
 $$\Gl\not\in [-1,2], \quad \mbox{or} \quad \Im {\Gl}\neq 0.$$
Since $\Gl$ satisfies $|P_N(\Gl)|=0$. There exists non-trial solution to the following linear equations
\beq
P_N(\Gl)^T \By=\mathbf{0},
\eeq
where $\By\in \RR^N$ and $\By\neq \mathbf{0}$. Next, for $\Bx=(\Bx_1,\Bx_2,\Bx_3)\in \RR^3$, define
\beq\label{layerstr01u01}
u=\left\{
\begin{split}
&a_{N}\Bx_1, \quad \Bx\in A_N,\\
&a_{j}\Bx_1 + b_{j}\frac{\Bx_1}{|\Bx|^3}, \quad \Bx\in A_j,\quad j=N-1, N-2, \ldots , 1\\
&b_{0}\frac{\Bx_1}{|\Bx|^3}, \quad \Bx\in A_{0},
\end{split}
\right.
\eeq
where $\mathbf{a}=(a_1, a_2, \ldots, a_N)$ and $\mathbf{b}=(b_0, b_1, \ldots, b_{N-1})$ are determined by
$$
\mathbf{a}=\Xi\By, \quad \mathbf{b}=\Xi^T\Upsilon_N\By,
$$
where $\Upsilon_N$ is defined in \eqnref{eq:matU01}.
It can then be verified that $u$ is the solution to
\beq\label{eq:mainmd01ap01}
\left\{
\begin{array}{ll}
\nabla\cdot \varepsilon\nabla u =0, & \mbox{in} \quad \RR^3\\
u=\Ocal(|\Bx|^{-1}), & |\Bx|\rightarrow \infty,
\end{array}
\right.
\eeq
with
\beq\label{eq:permidf0101}
\varepsilon(\Bx)=\varepsilon_c(\Bx)\chi(D)+\varepsilon_0\chi(\RR^3\setminus\overline{D}),
\eeq
and the parameter $\varepsilon_c$ is given by
\beq\label{eq:vepdef01ap01}
\varepsilon_c
=\left\{
\begin{array}{ll}
\frac{\Gl-2}{\Gl+1}\varepsilon_0, & \mbox{in}\quad A_j, \quad j \quad \mbox{is odd},\\
\varepsilon_0, & \mbox{in}\quad A_j, \quad j \quad \mbox{is even}.
\end{array}
\right.
\eeq
Note that $u$ is not identically zero in $\RR^3$. In fact, if $u\equiv0$, then there holds that
\beq\label{eq:eq0101}
\mathbf{a}+\Upsilon_N^{-1}M\mathbf{b}=0, \quad b_0=0,
\eeq
where
\beq\label{eq:defM01}
M=
\begin{bmatrix}
0 & 1 & 0 &\cdots & 0 \\
0 & 0 & 1 &\cdots & 0 \\
\vdots & \vdots &\vdots &\ddots & \vdots\\
0 & 0 & 0 & \cdots & 1\\
0 & 0  & 0 &\cdots & 0
\end{bmatrix}.
\eeq
Furthermore, by transmission conditions on the boundary of each layer, there holds,
\beq\label{eq:eq0102}
(I-M)\mathbf{b}=\Upsilon_N (I-M^T)\mathbf{a}.
\eeq
By combining \eqnref{eq:eq0101} and \eqnref{eq:eq0102}, one thus obtains that
\beq
(\Upsilon_N-M\Upsilon_N M^T)\mathbf{a}=\mathbf{0},
\eeq
which indicates that $\mathbf{a}=\mathbf{0}$. This is impossible since $\mathbf{a}=\Xi\By$ and $\By\neq \mathbf{0}$. Thus $u$ is not identically zero in $\RR^3$.

On the other hand, one can use integral by parts to compute
\beq
\int_{\RR^3}\varepsilon |\nabla u|^2=\sum_{j=0}^N\int_{A_j}\varepsilon |\nabla u|^2=0.
\eeq
But by using \eqnref{eq:vepdef01ap01} there also holds
\beq\label{eq:keyint01}
0=\int_{\RR^3}\varepsilon |\nabla u|^2=\varepsilon_0\sum_{j=0}^{\lfloor N/2 \rfloor}\int_{A_{2j}}|\nabla u|^2+ \frac{\Gl-2}{\Gl+1}\sum_{j=0}^{\lfloor (N+1)/2 \rfloor}\int_{A_{2j-1}}|\nabla u|^2.
\eeq
\begin{itemize}
  \item [Case i.] $\Gl\not\in [-1,2]$. It is obvious that
$$
\frac{\Gl-2}{\Gl+1}>0.
$$
Thus one immediately has $\nabla u=0$ in $\RR^3$ and thus $u=C$ in $\RR^3$, by the decay behavior of $u$ at infinity one thus show that $u\equiv0$ in $\RR^3$.
  \item [Case ii.] $\Im {\Gl}\neq 0$. Then by using \eqnref{eq:keyint01} one has
\beq\label{eq:keyint02}
\Im{\left(\frac{\Gl-2}{\Gl+1}\right)}\sum_{j=0}^{\lfloor (N+1)/2 \rfloor}\int_{A_{2j-1}}|\nabla u|^2=0.
\eeq
Thus $u=C$ in $A_{2j-1}$, $j=1, 2, \ldots \lfloor (N+1)/2 \rfloor$. By using integral by parts, one can readily show that $u=0$ in $A_0$ and thus $C=0$. Thus $u\equiv0$ in $\RR^3$.
\end{itemize}
We have shown that $u\equiv0$ either $\Gl\not\in [-1,2]$ or $\Im {\Gl}\neq 0$, which is in contradiction with our assumption. Thus \eqnref{eq:thmainsp01} holds.
\section*{Acknowledgement}
The work of Y. Deng was supported by NSFC-RGC Joint Research Grant No. 12161160314 and NSF grant of China No. 11971487.
The work of X. Fang was supported by  NSF grant of China No. 72001077, Humanities and Social Sciences Foundation of the Ministry of Education No. 20YJC910005, NSF grant of Hunan No. 2021JJ30192.

\section*{Conflict of interest statement}
The authors declared that they have no conflicts of interest to this work.
We declare that we do not have any commercial or associative interest that represents a conflict of interest in connection with the work submitted.

\section*{Data availability statement}
This is a mathematical paper containing all the necessary theoretical proofs and numerical illustrations.
There are no data to be reported concerning this work.


\begin{thebibliography}{1}
\bibitem{ACKLM1}
H. Ammari, G. Ciraolo, H. Kang, H. Lee, and G. Milton,  Spectral theory of a Neumann-Poincar\'e-type operator and analysis of cloaking due to anomalous localized resonance. Arch. Ration. Mech. Anal. (2) 208 (2013), 667--692.

\bibitem{ACKLM2}
H. Ammari, G. Ciraolo, H. Kang, H. Lee, and G. Milton, Anomalous localized resonance using a folded geometry in three dimensions. Proceedings of the Royal Society A, 469 (2013), 20130048.

\bibitem{ACKLM3}
H. Ammari, G. Ciraolo, H. Kang, H. Lee, G.W. Milton, Spectral theory
of a Neumann-Poincar\'e-type operator and analysis of anomalous localized resonance II, Contemporary Math., 615 (2014), 1--14.

\bibitem{ADKLMZ}
H. Ammari, Y. Deng and P. Millien, Surface plasmon resonance of nanoparticles and
applications in imaging,  Arch. Ration. Mech. Anal., 220(2016), 109--153.


\bibitem{AGJKLM12}
{H. Ammari, J. Garnier, V. Jugnon, H. Kang, H. Lee, and M. Lim},
Enhancement of near-cloaking. Part III: numerical simulations, statistical stability, and related questions.
Contemporary Mathematics, 577 (2012), 1-24.

\bibitem{AKLL11}
{H. Ammari, H. Kang, H. Lee, M. Lim}, {Enhancement of near cloaking using
generalized polarization tensors vanishing structures. Part I: The conductivity
problem}, Commun. Math. Phys., 317 (2013), 253-266.

\bibitem{AKLL13}
{H. Ammari, H. Kang, H. Lee, M. Lim},
Enhancement of near cloaking. Part II: the Helmholtz equation. Commun. Math. Phys.
, 317 (2013), 485-502.



%
%
%

\bibitem{BLBCL2006}
S. Berciaud, D. Lasne, G. A. Blab, L. Cognet and B. Lounis,
Photothermal Heterodyne Imaging of Individual Metallic
Nanoparticles: Theory versus Experiment, Phys. Rev. B
73 (2006), 045424.

\bibitem{BS11}
G. Bouchitt\'e and B. Schweizer, {\it Cloaking of small objects by anomalous localized resonance},
Quart. J. Mech. Appl. Math., {\bf 63} (2010), 437-463.

\bibitem{BL2002}
D. Boyer, P. Tamarat, A. Maali, B. Lounis, and M. Orrit, Photothermal Imaging of
Nanometer-Sized Metal Particles among Scatterers,
Science,  297 (2002), 1160--1163.



\bibitem{CKKL7}
D. Chung, H. Kang, K. Kim and H. Lee, {\it Cloaking due to anomalous localized resonance
in plasmonic structures of confocal ellipses}, SIAM J. Appl. Math., {\bf 74} (2014),
1691-1707.


\bibitem{DLZ21}
Y. Deng, H. Liu and G. Zheng, {\it Mathematical analysis of plasmon resonances for curved nanorods}, Journal de Math\'ematiques Pures et Appliqu\'ees, 153 (2021), 248--280.

\bibitem{DLZ22}
Y. Deng, H. Liu and G. Zheng, {\it Plasmon resonances of nanorods in transverse electromagnetic scattering}, Journal of Differential Equations, 318 (2022), 502--536.


\bibitem{FDC19}
X. Fang, Y. Deng and X. Chen, {\it Asymptotic behavior of spectral of Neumann-Poincar\'e operator in Helmholtz system}, Math. Meth. Appl. Sci., 42 (3) (2019), 942--953.

\bibitem{FDL15}
 X. Fang, Y. Deng and J. Li, {\it Plasmon resonance and heat generation in nanostructures}, Math. Meth. Appl. Sci., 38  (2015), 4663--4672.

%
%
%
%
%

\bibitem{HF04}
E. Hutter and J.H. Fendler, Exploitation of localized surface plasmon resonance,
Adv. Mater., {16} (2004), 1685--1706.

%

\bibitem{LL15}
H. Li and H. Liu,
{\it On anomalous localized resonance and plasmonic cloaking beyond the quasistatic
limit}, Proc. R. Soc. A, {\bf 474}: 20180165.

\bibitem{LLL9}
H. Li, J. Li and H. Liu, {\it On quasi-static cloaking due to anomalous localized resonance in $\mathbb{R}^3$},
SIAM J. Appl. Math., {\bf 75} (2015), 1245-1260.

\bibitem{LLL16}
H. Li, J. Li and H. Liu, {\it On novel elastic structures inducing polariton resonances with finite frequencies
and cloaking due to anomalous localized resonance}, J. Math. Pures Appl., {\bf 120} (2018), 195-219.



\bibitem{Pe99}
J. B. Pendry, Radiative Exchange of Heat Between Nanostructures, J. Phys.: Condens. Matter {11} (1999), 6621-6633.


\bibitem{RS} M. Ruiz and O. Schnitzer, {\it Slender-body theory for plasmonic resonance}, Pro. R. Soc. A, \textbf{475}:20190294.

\bibitem{SC10}
D. Sarid and W. A. Challener, \textsl{Modern Introduction to Surface Plasmons: Theory, Mathematica Modeling, and Applications},
Cambridge University Press, New York, 2010.

%

\bibitem{WN10}
G. Milton and N. Nicorovici, {\it On the cloaking effects associated with anomalous
localized resonance}, Proc. R. Soc. A, {\bf 462} (2006), 3027-3059.

%


\end{thebibliography}
\end{document}